\documentclass[12pt,twoside]{amsart}

\usepackage{amsmath, amssymb, amsthm, enumitem, esint}
\usepackage{hyperref}
\hypersetup{bookmarksdepth=3}

\usepackage[usenames,dvipsnames]{xcolor}

\definecolor{gray}{gray}{0.7}

\usepackage[margin=1.25in]{geometry}

\newcommand{\ben}{\begin{enumerate}}
\newcommand{\een}{\end{enumerate}}
\newcommand{\ra}{\rightarrow}
\newcommand{\R}{\mathbb R}

\newcommand{\bit}{\begin{itemize}}
\newcommand{\eit}{\end{itemize}}

\newcommand{\Ga}{\Gamma}

\newcommand{\D}{\partial}
\newcommand{\la}{\lambda}
\newcommand{\al}{\alpha}
\newcommand{\ol}{\overline}

\newcommand{\de}{\delta}

\newcommand{\lb}{\linebreak[1]}
\newcommand{\can}{\textnormal{can}}
\newcommand{\IR}{\mathbb{R}}

\newcommand{\IN}{\mathbb{N}}

\newcommand{\IZ}{\mathbb{Z}}

\newcommand{\M}{\mathcal{M}}

\newcommand{\eps}{\varepsilon}
\newcommand{\ov}[1]{\overline{#1}}
\newcommand{\td}[1]{\widetilde{#1}}

\DeclareMathOperator*{\osc}{osc}

\DeclareMathOperator{\rot}{rot}

\DeclareMathOperator{\Ric}{Ric}
\DeclareMathOperator{\Rm}{Rm}

\DeclareMathOperator{\id}{id}

\DeclareMathOperator{\spann}{span}

\newcommand{\EMPTY}[1]{}

\newtheorem*{Claim1}{Claim 1}
\newtheorem*{Claim2}{Claim 2}

\newtheorem{Theorem}{Theorem}[section]
\newtheorem{Lemma}[Theorem]{Lemma}

\newtheorem{Proposition}[Theorem]{Proposition}

\newtheorem{lemma}[Theorem]{Lemma}
\newtheorem{corollary}[Theorem]{Corollary}

\newtheorem{Definition}[Theorem]{Definition}
\newtheorem{definition}[Theorem]{Definition}
\theoremstyle{remark}

\numberwithin{equation}{section}

\title{On the rotational symmetry of 3-dimensional $\kappa$-solutions}
\author{Richard H. Bamler}
\address{Department of Mathematics, University of California, Berkeley, Berkeley, CA 94720}
\email{rbamler@berkeley.edu}
\author{Bruce Kleiner}
\address{Courant Institute of Mathematical Sciences, New York University,  251 Mercer St., New York, NY 10012}
\email{bkleiner@cims.nyu.edu}
\thanks{The first author was supported by  NSF grant DMS-1611906. \\
\hspace*{3mm} The second author was supported by NSF grants DMS-1405899, DMS-1406394, DMS-1711556, and a Simons Collaboration grant.}
\date{\today}

\begin{document}

\maketitle

\begin{abstract}
In a recent paper, Brendle showed the uniqueness of the Bryant soliton among 3-dimensional $\kappa$-solutions.
In this paper, we present an alternative proof for this fact and show that compact $\kappa$-solutions are rotational symmetric.
Our proof arose from independent work relating to our Strong Stability Theorem for singular Ricci flows.
\end{abstract}

\tableofcontents

\section{Introduction}
In his celebrated paper \cite{Perelman1}, Perelman characterized the singularity formation of 3-dimensional Ricci flows.
More specifically, he proved that singularities are always modeled on $\kappa$-solutions, which he classified in a qualitative way.
Roughly speaking, such solutions are either quotients of the round shrinking sphere or cylinder or they are diffeomorphic to $\IR^3, S^3$ or $\IR P^3$ and contain large regions that are asymptotically cylindrical.
A prominent example in the $\IR^3$-case is the Bryant soliton, which is rotationally symmetric.
In \cite{Perelman1}, Perelman conjectured that the Bryant soliton is the only $\kappa$-solution on $\IR^3$.

In a series of papers \cite{Brendle2013, Brendle2018}, Brendle proved this conjecture (see also related results on the mean curvature flow of Brendle and Choi \cite{Brendle2017, Brendle2018a}).
First, in \cite{Brendle2013} he proved uniqueness of the Bryant soliton under the additional assumption that the $\kappa$-solution is a soliton.
Second, in \cite{Brendle2018} he showed the following two theorems, which imply Perelman's conjecture.

\begin{Theorem} \label{Thm_Bryant_uniqueness}
Any rotationally symmetric $\kappa$-solution on $\IR^3$ is homothetic to the Bryant soliton.
\end{Theorem}

\begin{Theorem} \label{Thm_O3_invariance}
Any $\kappa$-solution on $\IR^3$ is rotationally symmetric.
\end{Theorem}

Here rotationally symmetric means that the solution admits an isometric $O(3)$-action whose principal orbits are 2-spheres. 
The proof of Theorem~\ref{Thm_O3_invariance} relies on the earlier uniqueness theorem from \cite{Brendle2013} and on Theorem~\ref{Thm_Bryant_uniqueness}.

In the same paper, Brendle remarks that the techniques in his paper can be adapted to the compact case, thereby proving:

\begin{Theorem} \label{Thm_O3_invariance_general}
Any 3-dimensional $\kappa$-solution is either rotationally symmetric or homothetic to a quotient of the round shrinking sphere.
\end{Theorem}

In this paper we offer an alternative proof of Theorem~\ref{Thm_O3_invariance_general}.  This proof arose from independent work on questions relating to singular Ricci flows and their strong stability properties.  Our primary motivation is to share this different approach with the community since the ideas may be useful in other contexts.  We also want to provide a detailed argument covering both the compact and noncompact cases.  We emphasize that we acknowledge Brendle's prior solution to this problem and note, moreover, that our argument relies on Brendle's Theorem~\ref{Thm_Bryant_uniqueness}.

We now give a brief sketch of our argument.
Like Brendle's proof, our proof relies on a stability result that states that the degree of rotational symmetry improves as we move forward in time.
We establish this stability property in two steps.
In Section~\ref{sec_linear}, we first consider the linearized problem and show that rotational symmetry is stable under the flow modulo a few modes, which can be removed by reparameterization.
In Section~\ref{sec_main_argument}, we use a limit argument and the Strong Stability Theorem for Ricci flow spacetimes from our paper \cite{BamlerKleiner2017} to reduce the non-linear case to the linear case.

Both steps of our proof are different from Brendle's approach.
For example, we employ a different iteration scheme that allows us to avoid having to localize several estimates in the linear and non-linear cases.
As a result we don't have to deal with error terms arising from the boundary.
In addition, we use of the Strong Stability Theorem to directly compare metrics with rotationally symmetric ones.
This approach replaces Brendle's analysis of approximate Killing fields.

\section{Preliminaries}

\subsection{The Ricci-DeTurck equation and its linearization}
We briefly recall the Ricci-DeTurck equation.
For more details we refer to \cite[Appendix~A]{BamlerKleiner2017} (where the same notation is used) or \cite{Topping2006}.
Let $(g_t)_{t \in [t_1,t_2]}$ and $(g'_t)_{t \in [t_1,t_2]}$ be Ricci flows on a manifolds $M, M'$, respectively and consider a family of diffeomorphisms $(\chi_t :M' \to M)_{t \in [t_1,t_2]}$, evolving by the harmonic map heat flow
\[ \partial_t \chi_t = \triangle_{g'_t, g_t} \chi_t = \sum_{i=1}^n \big( \nabla^{g_t}_{d\chi_t (e_i)} d\chi_t (e_i) - d\chi_t (\nabla^{g'_t}_{e_i} e_i) \big), \]
where $\{ e_i \}_{i=1}^n$ is a local frame field on $M'$ that is orthonormal with respect to $g'_t$.
Then the pullback
\[ g_t + h_t := \big( \chi_t^{-1} \big)^* g'_t - g_t \]
satisfies the \textbf{Ricci-DeTurck equation}
\begin{equation} \label{eq_RDT}
 \partial_t (g_t + h_t) = -2 \Ric (g_t + h_t) - \mathcal{L}_{X_{g_t} (g_t + h_t)} (g_t + h_t), 
\end{equation}
where the vector field $X_{g_t} (g_t + h_t)$ is defined by
\[ X_g (g^*) := \triangle_{g^*, g} \id_M = \sum_{i=1}^n \big( \nabla^g_{e_i} e_i - \nabla^{g^*}_{e_i} e_i \big), \]
for a local frame $\{ e_i \}_{i=1}^n$ that is orthonormal with respect to $g^*$.
The Ricci-DeTurck equation has the following analytical structure
\begin{equation} \label{eq_RDT_analytical}
 \nabla_{\partial_t} h_t = \triangle_{g_t} h_t + 2 \Rm_{g_t} (h_t) + \mathcal{Q}_{g_t} [h_t] , 
\end{equation}
where the left-hand side uses Uhlenbeck's trick
\[  (\nabla_{\partial_t} h_t)_{ij} = (\partial_t h_t)_{ij}  - \frac12 g^{pq} \big( h_{pj}  \partial_t g_{qi} + h_{ip}  \partial_t g_{qj} ), \]
\[ (\Rm_{g_t} (h_t))_{ij} = g^{pq} R_{p ij}^{\quad u} h_{qu} \]
and the last term has the structure
\[ \mathcal{Q}_{g_t} [h_t] = \big( (g+h)^{-1} - g^{-1} \big) * \big( \nabla^{g_t, 2} h + \Rm_{g_t} * h_t ) + (g+h)^{-1} * (g+h)^{-1} *  \nabla^{g_t} h * \nabla^{g_t} h. \]
The linearization of (\ref{eq_RDT_analytical}) is called the \textbf{linearized Ricci-DeTurck equation}
\begin{equation} \label{eq_LRDTF}
 \nabla_{\partial_t} h_t = \triangle_{g_t} h_t + 2 \Rm_{g_t} (h_t).
\end{equation}
The following fact, which has also been used in \cite{Brendle2013, Brendle2018}, will be important for us.

\begin{lemma} \label{lem_HessianLRDTF}
If $w :M\times [t_1,t_2]\ra \R$ solves the heat equation $\partial_t w = \triangle_{g_t} w$ on a Ricci flow background $(M,(g_t)_{t\in [t_1,t_2]})$, then its Hessian $h_t:=\nabla^2 w_t$ solves (\ref{eq_LRDTF}) on the same Ricci flow background.
\end{lemma}

\begin{proof}
Let $(X_t)_{t \in [t_1,t_2]}$ be a time dependent vector field that evolves by the heat equation 
\begin{equation} \label{eq_HEX}
\nabla_{\partial_t} X = \triangle_{g_t} X.
\end{equation}
Then in any orthonormal frame
\begin{multline*}
 \nabla_{\partial_t} \nabla_i X^j - \nabla_i \nabla_{\partial_t} X^j =  \dot \Gamma_{ik}^j \, X^k +  R_{ik} \, \nabla_k X^j + \nabla_i R_{jk} \, X^k \\
 =  - \nabla_k R_{ij} \, X^k + \nabla_j R_{ik} \, X^k + R_{ik} \, \nabla_k X^j,
\end{multline*}
while
\[ \triangle \nabla_i X^j - \nabla_i \triangle X^j = - 2R_{iklj} \, \nabla_k X^l + R_{ik} \, \nabla_k X^j - \nabla_l R_{ilkj} \,  X^k. \]
Combining both equations and applying the second Bianchi identity yields
\[  \nabla_{\partial_t} \nabla_i X^j - \triangle \nabla_i X^j = - 2R_{iklj} \, \nabla_k X^l . \]
So $h_t := \mathcal{L}_{X_t} g_t$ solves (\ref{eq_LRDTF}).
Lastly, observe that by the Bianchi identity $X := \frac12 \nabla^{g_t} w$ solves (\ref{eq_HEX}).
So $h_t = \frac12 \mathcal{L}_{\nabla w_t} g_t = \nabla^2 w_t$ solves (\ref{eq_LRDTF}).
\end{proof}

\bigskip

\subsection{Linearized Ricci-DeTurck flow on the round cylinder}
\label{subsec_lrdtf_round_cylinder}
We now show that bounded ancient solutions to the linearized Ricci-DeTurck flow on the round cylinder must be Hessians of a very special form.  
This will be used in the proof of Proposition~\ref{Prop_LRDTF_0T}.
We remark that the following results are similar to \cite[Proposition 6.1]{Brendle2018}.

In order to facilitate the proof of decay using eigenspace decompositions, we will make use of appropriate $L^2$-type norms.
\begin{definition}[Fiberwise $L^2$-norm]
\label{def_fiberwise_l2_norm}
Let $(S^2\times \R,(g_t)_{t\in(-\infty,0]})$ be the shrinking round cylinder, and suppose  $(h_t)_{t\in I}$ is a $2$-tensor field defined on a time interval $I$.  Then for $t\in I$,  the {\bf (normalized) fiberwise $L^2$-norm} of $h_t$ at $r\in \R$ is 
$$
\|h_t\|_{L^2(S^2\times\{r\})}:=\left(\frac{1}{|S^2\times\{r\}|}\int_{S^2\times \{ r \}} |h_t|^2\;dV  \right)^\frac12
$$
where $dV$ and $|h|$ denotes the Riemannian measure and Riemannian norm induced by $g_t$, respectively.
\end{definition}

\begin{lemma}[Partial vanishing on the cylinder]
\label{lem_partial_vanishing_cylinder}
Let $(h_t)_{t\in I}$ be a linearized Ricci-DeTurck flow on a shrinking round cylinder $(S^2\times\R,(g_t)_{t\in(-\infty,0)})$, with $g_t=dr^2+2|t|g_{S^2}$.  
Assume that the average of $h$ under the standard $O(3)$-action vanishes.
Then
\begin{enumerate}[label=(\alph*)]
\item $\sup_{r\in \R}\;\|h_t\|_{L^2(S^2\times\{r \})}$ is a non-increasing function of $t$.
\item If $\inf I=-\infty$ and $\sup_{(r,t)\in\R\times I}\|h_t\|_{L^2(S^2\times\{ r \})}<\infty$, then 
\[ h=(a_1 u_1 + a_2 u_2 + a_3 u_3) \,g_{S^2}, \]
where $u_1, u_2, u_3$ are the coordinate functions on $S^2 \subset \IR^3$ and $a_1, a_2, a_3 \in \IR$.
\end{enumerate}
\end{lemma}
\begin{proof}
This follows from separation of variables and the maximum principle.

The linearized Ricci-DeTurck equation has the form
\begin{equation}
\label{eqn_lrdtf_cylinder}
\begin{aligned}
\nabla_{\D_t}h_t&=\Delta_th_t+\Rm_t(h_t)=\Delta^{S^2}_th_t+ \nabla^2_{\D_r, \D_r} h_t+\Rm_t(h)\\
&= \nabla^2_{\D_r, \D_r} h_t+|t|^{-1}(\Delta^{S^2}_{-1}h_t+\Rm_{-1}(h))\,.
\end{aligned}
\end{equation}We have two decompositions of the space $\Ga(s^2T^*(S^2\times \R))$ of symmetric $2$-tensor fields:
the decomposition  
$$
\Ga(s^2T^*(S^2\times \R))=\Ga(s^2\R^*)\oplus \Ga(s^2(T^*S^2))\oplus \Ga(T^*(S^2)\otimes T^*\R)
$$
induced by the bundle decomposition $s^2(T^*(S^2\times\R))\simeq s^2\R^*\oplus s^2(T^*S^2)\oplus (T^*(S^2)\otimes T^*\R)$, and the decomposition $h=\sum_jh_j$ induced by the eigenspace decomposition for the fiber Laplacian $\Delta^{S^2}_{-1}$. 
Straightforward computation shows that these decompositions are compatible with one another, and also with both the linearized Ricci-DeTurck flow and the fiberwise $L^2$-metric.
So it suffices to verify the lemma when $h$ lies in a single summand of each of the decompositions.

{\em Case 1. $h\in \Ga(s^2T^*\R)\oplus\Ga(T^*S^2\otimes T^*\R)$ belongs to the $\la$-eigenspace of $\Delta^{S^2}_{-1}$. \quad}  Then $\Rm(h)\equiv 0$ because $i_{\D_r}\Rm= 0$, and $\la>0$ since 
the zero eigenspace of $\Delta^{S^2}_{-1}$ intersects $\Ga(s^2T^*\R)\oplus\Ga(T^*S^2\otimes T^*\R)$ in the rotationally symmetric tensors, which vanish by assumption.  
The maximum principle applied to (\ref{eqn_lrdtf_cylinder}) now gives assertion (a), and if $I=(-\infty,t_0)$ then $h\equiv 0$.

{\em Case 2. $h\in \Ga(s^2(T^*S^2))$ belongs to the $\la$-eigenspace of $\Delta^{S^2}_{-1}$. \quad}  Then $h$ further decomposes as $h=h^{\text{scalar}}+h^{\text{traceless}}$ where $h^{\text{scalar}}=\phi g_{S^2}$ and $h^{\text{traceless}}$ is traceless.
Furthermore, this decomposition is  invariant under linearized Ricci-DeTurck flow.  
If $h = h^{\text{traceless}}$, then $\Rm(h)=-\frac12 h$, so applying the maximum principle as in Case 1 we are done.  
If $h=h^{\text{scalar}}$, then $\Rm(h)=\frac12 h$.   
If $\la< -1$, then we are done by the maximum principle.  
If $\la\geq -1$, then  $\la=-1$, because the case $\la =0$ is excluded by the fact that the $O(3)$-average of $h$ vanishes.   
Hence (\ref{eqn_lrdtf_cylinder}) reduces to the direct sum of three copies of the standard heat equation.  Now assertion (a) follows from the maximum principle, while if $I=(-\infty,t_0)$, then $\nabla_{\partial_r} h\equiv \D_t h\equiv 0$ by a gradient estimate.
\end{proof}

\subsection{A semilocal maximum principle}
In this subsection we restate the semilocal maximum principle from \cite[Proposition~9.1]{BamlerKleiner2017} in a slightly different form for the case in which the background flow is a $\kappa$-solution and the perturbation $h$ evolves by the linearized Ricci-DeTurck flow.

\begin{Proposition}[Semilocal maximum principle] 
\label{Prop_weighted_semilocal_max}
For any $E > 1$ there are constants $L = L(E), H = H(E), C= C(E) < \infty$ such that the following holds.

Let $(M, (g_t)_{t \leq 0})$ be a $\kappa$-solution, $T < 0$ and $(x_0,t_0) \in M \times [-T, 0]$.
Consider the parabolic neighborhood $P_L(x_0,t_0) := P(x_0,t_0, L R^{-1/2} (x_0,t_0))$ and let $(h_t)_{t \in [-T,0]}$ be a solution to the linearized Ricci-DeTurck flow equation (\ref{eq_LRDTF}) on $P_L(x_0,t_0) \cap M \times [-T,0]$. 
Then for any $a \geq 0$ we have
\begin{multline} \label{eq_semi_loc}
 \bigg( e^{- H a t} \frac{|h|}{ R^{E}+a^{E}} \bigg) (x_0,t_0) \leq \frac1{100} \sup_{P_L(x_0,t_0) \cap M \times [-T,0]} e^{- H a t} \frac{|h|}{R^{E}+a^{E}} \\
 + C \sup_{P_L (x_0,t_0) \cap M \times \{ - T \}} e^{- H a t} \frac{|h|}{R^{E}+ a^{E}} .
\end{multline}
f $P_L (x_0,t_0) \cap M \times \{ - T \} = \emptyset$, then the last term can be omitted.
\end{Proposition}

\begin{proof}
The proof is similar to that of \cite[Proposition~9.1]{BamlerKleiner2017} and essentially follows by rescaling the factor $a$.
For convenience of the reader we provide a proof here.

After applying a time-shift and parabolic rescaling we may assume without loss of generality that $t_0 = 0$ and $R(x_0, 0) = 1$.
Fix $E > 1$ and $L_i, H_i, C_i \to \infty$ and consider a sequence of counterexamples $(h_{i,t})_{t \in [-T_i,0]}$, $(M_i, (g_{i,t})_{t \leq 0})$, $(x_i, 0) \in M_i \times [-T_i,0]$, $a_i \geq 0$ with $R(x_i, 0)=1$.
After multiplying $(h_{i,t})_{t \in [-T_i,0]}$ by a scalar, we may assume that
\begin{equation} \label{eq_hixi_is_1}
 |h_i| (x_i, 0) = 1. 
\end{equation}
Next, we may assume that the $(M_i, (g_{i,t})_{t \leq 0})$ are $\kappa$-solutions for some uniform $\kappa > 0$, because otherwise the flows $(M^i, (g_{i,t})_{t \leq 0}$ would be quotients of the round shrinking sphere for large $i$ (see \cite[Lemma~C.1(a)]{BamlerKleiner2017}) and we could pass to the universal covers.
So, after passing to a subsequence, we may assume that the pointed flows $(M_i, (g_{i,t})_{t \leq 0}, x_i)$ smoothly converge to a pointed $\kappa$-solution $(M_\infty, (g_{\infty, t})_{t \leq 0}, x_\infty)$.
Lastly, after passing to a subsequence, we may assume that the limits $a_\infty := \lim_{i \to \infty} a_i \in [0, \infty]$ and $T_\infty := \lim_{i \to \infty} T_i \in [0, \infty]$ exist.

The assumption that the tensor fields $(h_{i,t})_{t \in [-T_i,0]}$ are counterexamples to (\ref{eq_semi_loc}) implies
$$ \sup_{P_{L_i} (x_i, 0) \cap M_i \times [-T_i,0]} e^{-H_i a_i t} \frac{|h_i|}{R^E+a_i^E} \leq 100 \frac{1}{1 + a_i^E} $$
$$ \sup_{P_{L_i} (x_i, 0) \cap M_i \times \{ -T_i \}} e^{-H_i a_i t} \frac{|h_i|}{R^E+a_i^E} \leq C_i^{-1}  \frac{1}{1 + a_i^E} $$
We can now argue as in the proof of \cite[Proposition~9.1]{BamlerKleiner2017} that $T_\infty = \infty$ and that, after passing to a subsequence, the tensor fields $(h_{i,t})_{t \in [-T_i,0]}$ converge to a solution $(h_{\infty,t})_{t \in (-\infty,0]}$ of the linearized Ricci-DeTurck equation on $(M_\infty, (g_{\infty,t})_{t \leq 0})$.
If $a_\infty > 0$, then $\lim_{i \to \infty} e^{-H_i a_i t} = 0$ for all $t < 0$ and therefore $h_\infty \equiv 0$, which contradicts (\ref{eq_hixi_is_1}).
On the other hand, if $a_\infty = 0$, then (\ref{eq_hixi_is_1}) implies $|h_\infty| \leq C' R^E$ for some $C' < \infty$.
Together with the Vanishing Theorem \cite[Theorem~9.8]{BamlerKleiner2017}, this implies $h_\infty \equiv 0$, again contradicting (\ref{eq_hixi_is_1}).
\end{proof}

\begin{corollary}
\label{cor_interior_decay}
For every $1<E<\infty$, $\de>0$ there is an $L'<\infty$ such that if $(M,(g_t)_{t \leq 0})$ is a $\kappa$-solution, $(h_t)$ is a linearized Ricci-DeTurck flow on the parabolic ball $P_L (x_0,t_0):=P(x_0,t_0,LR^{-1/2}(x_0,t_0))$, then 
$$
R^{-E}|h|(x_0,t_0)\leq \de\sup_{P_L(x_0,t_0)} R^{-E}|h|\,.
$$
\end{corollary}
\begin{proof}
This follows by iterating Proposition~\ref{Prop_weighted_semilocal_max} with $a=0$.
\end{proof}

\section{A Partial Vanishing Theorem for the linearized Ricci-DeTurck flow on $\kappa$-solutions} \label{sec_linear}
In the following we will consider 3-dimensional rotationally symmetric $\kappa$-solutions $(M, (g_t)_{t \leq 0})$, i.e. solutions that are invariant under an $O(3)$-action whose principal orbits are 2-spheres.
Our goal will be to analyze the linearized Ricci-DeTurck flow on these solutions and to deduce that this flow decays modulo certain well-understood modes.

As $(M, (g_t)_{t \leq 0})$ is assumed to be rotationally symmetric, the possible topological types of $M$ are $S^2 \times \IR, \IR^3, S^3, \IR P^3$.
In this section, we will only focus on the non-compact cases, i.e. the cases $M \approx \IR^3$ or $S^2 \times \IR$.
Here we equip $\IR^3$ and $S^2 \times \IR$ with the standard $O(3)$-action.
It is a well known fact that in the case $M \approx S^2 \times \IR$, the flow is  homothetic to the round shrinking cylindrical flow $g_t = dr^2 + 2|t| g_{S^2}$.

We can express $g_t$ as a warped product of the following form, away from the center of rotation if $M \approx \IR^3$:
\begin{equation} \label{eq_WP}
  g_t = p^2(r,t) dr^2 + q^2 (r,t) g_{S^2}. 
\end{equation}
The symmetric $(0,2)$-tensors $h$ that are invariant under the $O(3)$-action always take the similar form
$$ h = \td{p}^2(r) dr^2 + \td{q}^2 (r) g_{S^2}. $$
We will refer to such tensors from now on as \emph{rotationally invariant}.
In the following we will also consider the three coordinate functions $u_1, u_2, u_3 \in C^\infty (S^2)$ for the standard embedding $S^2 \subset \IR^3$.
We will often view these functions as smooth functions on $S^2 \times \IR$ or $\IR^3 \setminus \{ 0\}$.
So $u_1, u_2, u_3$ are constant in time and along radial geodesics.

The following proposition is the main result of this section.
It is similar to the Vanishing Theorem \cite[Theorem 9.8]{BamlerKleiner2017}.
The main difference is that we only assume uniform bounds on $h$ on the initial time-slice, without any weight.
As a result, we can only control $h$ at later times modulo certain modes, which are either rotationally invariant or can be expressed as the Hessian of a scalar function.

\begin{Proposition} \label{Prop_LRDTF_0T}
Let $(M, (g_t)_{t\leq 0})$ be a rotationally symmetric $\kappa$-solution diffeomorphic to $\IR^3$ or $S^2 \times \IR$ and let $m \in \IZ$, $\eta > 0$ and $C, D < \infty$.
Then there is a constant $T = T(m, \eta, C, D, (g_t)) < \infty$ such that the following holds.

Let $(h_t)_{t \in [-T,0]}$ be a uniformly bounded solution to the linearized Ricci-DeTurck flow (\ref{eq_LRDTF}) on  $(M, (g_t)_{t\leq 0})$ and assume that $R(x,0) = 1$ for some $x \in M$ .
Assume that $|\nabla^m h_{-T}| \leq CR^{m/2}$ on $M$ for all $m = 0, \ldots, 3$.
Then on $B(x,0, D)$ we have a decomposition of the form  
$$ h_0 = h^{\rot}_0 + \nabla^2 \big( f_1(r) u_1 + f_2(r)u_2 + f_3(r) u_3 \big) + h'_0, $$
where:
\begin{enumerate}[label=(\alph*)]
\item $h^{\rot}_0$ is rotationally invariant.
\item $f_1(r) u_1 + f_2(r)u_2 + f_3(r) u_3$ is smooth on $B(x,0,D)$.
This implies in particular that  $f_1(r), \lb f_2(r), \lb f_3(r)$ vanish at the origin in the $\IR^3$-case.
\item $\Vert h'_0 \Vert_{C^{m}(B(x,0,D))} \leq \eta$.
\end{enumerate}
\end{Proposition}

The proof of Proposition~\ref{Prop_LRDTF_0T} uses the following fact:
\begin{equation} \label{eq_extra_assumption}
 \lim_{t \to -\infty} |t| \max_M R(\cdot,t) = \infty \qquad \text{if} \qquad M \approx \IR^3  . 
\end{equation}
This fact holds due to Theorem~\ref{Thm_Bryant_uniqueness}, which is due to Brendle.
We remark that with some extra work it is possible to remove the dependence on (\ref{eq_extra_assumption}), hence making the proof of Theorem~\ref{Thm_O3_invariance_general} independent of Theorem~\ref{Thm_Bryant_uniqueness}.

Before carrying out the proof of Proposition~\ref{Prop_LRDTF_0T}, we first introduce some general terminology and establish some preliminary lemmas.
For the remainder of this section, we will always assume that we are in the setting of Proposition~\ref{Prop_LRDTF_0T}.

Averaging via the isometric $O(3)$-action yields a decomposition
\[ h_t = h^{\rot}_t + h^{\osc}_t, \]
where $h^{\rot}_t$ is rotationally invariant and the average of $h^{\osc}_t$ under the $O(3)$-action vanishes.
Both components still solve the linearized Ricci DeTurck flow equation.
It therefore remains to prove Proposition~\ref{Prop_LRDTF_0T} in the case in which $h_t = h^{\osc}_t$, with the additional assertion that $h^{\rot}_0 \equiv 0$.

Next, we find a decomposition of $h_t$  of the form
\begin{equation} \label{eq_second_dec}
 h_t = h^{3\text{d},1}_t + h^{3\text{d},2}_t + h^{3\text{d},3}_t + h^{\text{rest}}_t. 
\end{equation}
Here the first three terms describe the component of $h^{\osc}_t$ corresponding to the 3-dimensional representation of $O(3)$.
More specifically, assume that the $O(3)$-action on $M$ is described by the family of diffeomorphisms $(\phi_A : M \to M)_{A \in O(3)}$.
Then for $j = 1,2,3$ we define $h^{3\text{d},j}_t$ to be the image of $h_t$ under the projection
\begin{equation} \label{eq_projection_j}
 h \longmapsto \frac1{|O(3)|} \int_{O(3)} \langle A \vec{e}_j, \vec{e}_j \rangle \, \phi_A^* h \, dA, 
\end{equation}
where we integrate with respect to a bi-invariant measure on $O(3)$.
We also set $h^{\text{rest}}_t := h_t - \sum_{j=1}^3 h^{3\text{d},j}_t$.
Then $h^{3\text{d},j}_t$ and $h^{\text{rest}}_t$ solve the linearized Ricci-DeTurck equation and the image of $h^{\text{rest}}_t$ under the projections (\ref{eq_projection_j}) vanishes.
Due to the decomposition (\ref{eq_second_dec}) it suffices to prove Proposition~\ref{Prop_LRDTF_0T} separately for the following two cases:
\begin{enumerate}[label=(\Alph*)]
\item $h_t^{\rot} \equiv 0$ and $h_t = h^{3\text{d},j}_t$ for some $j = 1, 2, 3$.
\item $h_t^{\rot} \equiv 0$ and $h_t = h^{\text{rest}}_t$.
\end{enumerate}
So let us assume without loss of generality that case (A) or (B) holds.
Case (A) will be more subtle and we will mostly focus on this case.
Case (B) will follow along the lines with small modifications and omissions of several technical details.
We will point out these differences in the course of the proof.

Let us now consider Case (A).
We first need to analyze the structure of the components of $h_t = h^{3\text{d},j}_t$ more carefully.
For this purpose, fix some $t \in [-T,0]$ and reparameterize the radial parameter $r$ such that the representation (\ref{eq_WP}) simplifies to $g_t = dr^2 + q^2(r) g_{S^2}$.
Let $\mu_j := q \, du_j$ and $\nu_j := q \, (*d u_j)$, where the star operator is taken fiberwise with respect to the standard metric on $S^2$.
Note that the maximum of $|\mu_j|_{g_t} = |\nu_j|_{g_t}$ on each cross-sectional 2-sphere is equal to 1.

\begin{Lemma}
We have on $S^2 \times \IR$ or $\IR^3 \setminus \{ 0 \}$
\begin{equation} \label{eq_characterization}
 h_t = h_t^{\textnormal{3d},j} = a_j(r) u_j \, dr^2 + b_j (r) u_j \, g_{S^2} + c_j(r) \, (\mu_j \, dr + dr \, \mu_j) + d_j (r) \, (\nu_j \, dr + dr \, \nu_j), 
\end{equation}
for some smooth radial functions $a_j(r), b_j(r), c_j(r), d_j(r)$, which extend to smooth odd functions across the origin if $M \approx \IR^3$.
\end{Lemma}

\begin{proof}
In the following, we will omit the index $t$.
It suffices to verify the characterization (\ref{eq_characterization}) along a single $S^2$-fiber.
Along such a fiber we can write $h = f dr^2 + ( \xi \, dr + dr \, \xi) + h^{\Vert}$, where $f \in C^\infty (S^2)$, $\xi$ is a 1-form and $h^{\Vert}$ is a symmetric 2-tensor on $S^2$.
Note that $f$, $\xi$ and $h^{\Vert}$ are contained in the image of the projection (\ref{eq_projection_j}), where $\phi_A$  denotes the standard action on $S^2$ and the pullback has to be taken within the appropriate category.
It remains to prove that $f$ is a multiple of $u_j$, $\xi \in \spann \{ \mu_j, \nu_j \}$ and $h^{\Vert}$ is a multiple of $u_j g_{S^2}$.
This follows from standard representation theory.

More specifically, consider a 3-dimensional representation of $O(3)$ of the form $\spann \{ \tau_1, \tau_2, \tau_3 \}$, where $\tau_j$ are tensor fields on $S^2$ of arbitrary degree.
Assume that the $\tau_j$ are chosen so that there is an equivariant map $\spann \{ \tau_1, \tau_2, \tau_3 \} \to \IR^3$, with $\tau_j \mapsto e_j$, where $O(3)$ acts on $\IR^3$ in the standard way.
Then $\tau_j$ must be invariant by rotations along the $\IR e_j$-axis and $\tau_j$ restricted to any great circle passing through $\pm e_j$ must satisfy the ODE $\tau''_j = - \tau_j$.
If $\tau_j$ is a scalar function or a symmetric $2$-tensor, then restricted to any great circle passing through $\pm \tau_j$ it must be even across $\pm e_j$. 
If $\tau_j$ is a 1-form, then it must be odd.
These properties uniquely determine $\tau_j$ up to a multiplicative constant.
\end{proof}

Given the coefficient functions $a_j(r), b_j(r), c_j(r), d_j(r)$, we define
\[ \mathcal{F}_j h_t = \frac{(\partial_r q) q^2c_j-b_j}{1-(\partial_r q)^2}. \]
Note that $\mathcal{F}_j$ is a zeroth order linear operator on $S^2 \times \IR$ or $\IR^3 \setminus \{0 \}$, respectively and $\mathcal{F}_j h_t$ is a smooth radial function where defined.

\begin{Lemma} \label{eq_Fhu_smooth}
The product $(\mathcal{F}_j h_t) u_j$ extends to a smooth radial function on $M$.
\end{Lemma}

\begin{proof}
If $M \approx \IR^3$, then $b_j (r)$ is an odd function that vanishes to at least second order.
Moreover, since $(M,g_t)$ has strictly positive sectional curvature, $\partial^3_r q (0) \neq 0$.
It follows that $\mathcal{F}_j h_t$ extends to a smooth odd function across the origin, which implies the statement of the lemma.
\end{proof}

\begin{Lemma}
\label{lem_cal_f_recovers_f}
For any smooth radial function $f(r)$ that extends to a smooth odd function  across the origin when $M \approx \IR^3$, the Hessian $\nabla^2 (f u_j)$ is of the form (\ref{eq_characterization}) and we have
\[ \mathcal{F}_j \big( \nabla^2 (  fu_j)  \big) = f. \]
\end{Lemma}

\begin{proof}
An elementary computation shows that
\begin{multline*}
 \nabla^2(fu_j)=(\partial^2_r f) u_j\,dr^2+(-fq^{-2}+(\partial_r q) q^{-1} \partial_r f)q^2 \, u_jg_{S^2} \\
+((\partial_r f) q^{-1}-f(\partial_ r q) q^{-2})\,(\mu_j \, dr+dr \, \mu_j). \qedhere
\end{multline*}
\end{proof}

\bigskip
Motivated by the previous lemma, we define in Case (A)
\[ \alpha[h_t] := \Big| h_t - \nabla^2  \big( ( \mathcal{F}_j h_t ) u_j  \big) \Big|. \]
So $\alpha[h_t]$ measures the deviation of $h_t$ from being a Hessian of a specific form.
In Case (B), we simply set $\alpha [ h_t ] := |h_t|$.

\begin{Lemma} \label{Lem_alpha_plus_Hessian}
Assume that we are in Case (A) and assume that $\alpha[h_{t^*}]$ is uniformly bounded for some $t^* \in [-T, 0]$.
Then $(\mathcal{F}_j h_{t^*} ) u_j$ is a smooth scalar function on $M$ that grows at most quadratically at infinity.
Let $(w_t)_{t \in [t^*, 0]}$ be the solution to the heat equation with initial condition $w_{t^*} = (\mathcal{F}_j h_{t*} ) u_j$.
Then
\[ \td{h}_t := h_t - \nabla^2 w_t \]
is a uniformly bounded solution to the linearized Ricci-DeTurck flow with $\alpha[\td{h}] = \alpha[h]$ on $M \times [t^*, 0]$ and $|\td{h}_{t^*}| =\alpha [\td{h}_{t^*}]$.
\end{Lemma}

\begin{proof}
The fact that $w_{t^*} = (\mathcal{F}_j h_{t^*} ) u_j$ is smooth follows from Lemma~\ref{eq_Fhu_smooth}.
Next, note that $|\nabla^2 w_{t^*}| \leq |h_{t*}| + \alpha[h_{t^*}]$ is uniformly bounded.
So $w_{t^*}$ grows at most quadratically at infinity and $|\nabla^2 w|$ remains uniformly bounded on $M \times [t^*,0]$.
Lemma~\ref{lem_HessianLRDTF} implies that $h_{t} - \nabla^2 w_t$ solves the linearized Ricci-DeTurck flow.
Next, Lemma~\ref{lem_cal_f_recovers_f} yields
\[ (\mathcal{F}_j \td{h}_t)u_j = \big( \mathcal{F}_j (h_t - \nabla^2 w_t)\big)u_j = (\mathcal{F}_j h_t)u_j - w_t,   \]
which implies that $\alpha[\td{h}_t] = |\td{h}_t - \nabla^2 ((\mathcal{F}_j h_t) u_j) + \nabla^2 w_t| = \alpha[h_t]$.
The last statement follows by definition of $\alpha$.
\end{proof}

\begin{lemma}
\label{lem_l_large_alpha_small}
For every $\de>0$ there is a $\Theta=\Theta (\de)<\infty$ such that if $(h_t)_{t\in [-\Theta,0]}$ is a bounded linearized Ricci-DeTurck flow on the shrinking round cylinder  $(S^2\times\R,(g_t)_{t\leq 0})$ with $R(\cdot,0)=1$ and $h^{\rot} \equiv 0$, then $\al[h](x,0)\leq \de \sup_M |h_{-\Theta}|$.  
\end{lemma}
\begin{proof}
Suppose not.  
Then for some $\de>0$, there is a sequence $\Theta_i\ra\infty$ and for every $i$ a linearized Ricci-DeTurck flow $(h_{i,t})_{t\in [-\Theta_i,0]}$ such that $\al[h_i](x,0)\geq \de $ and $\sup_M |h_{i,-\Theta_i}|\leq 1$.
Using the fiberwise $L^2$-norm (see Definition~\ref{def_fiberwise_l2_norm}), we have
\begin{equation}
{\label{eqn_tilde_h_l2}
}\sup_{r\in\R}\| h_{i,-\Theta_i}\|_{L^2(S^2\times\{r\})}\leq \sup_{S^2\times\R}|h_{i,-\Theta_i}|\leq 1,
\end{equation}
so by assertion (a) of Lemma~\ref{lem_partial_vanishing_cylinder} we have $\sup_{(r,t)\in\R\times [-\Theta_i,0]}\| h_{i,t}\|_{L^2(S^2\times\{r\})}\leq 1$.  
By (\ref{eqn_tilde_h_l2}) we may extract a limiting linearized Ricci-DeTurck flow $( h_{\infty,t})_{t\leq 0}$ such that 
\begin{equation}
\label{eqn_h_infty_l2}
\al[h_\infty](x,0)\geq \de \qquad\text{and}\qquad\sup_{(r,t)\in\R\times(-\infty,0]}\|h_{\infty, t}\|_{L^2(S^2\times\{r\})}\leq 1\,.
\end{equation}
By assertion (b) of Lemma~\ref{lem_partial_vanishing_cylinder} and Lemma~\ref{lem_cal_f_recovers_f} we would have $\nabla^2 ( \sum_{j=1}^3   ( \mathcal{F}_j h_{\infty,0} ) u_j ) =h_{\infty,0}$, and hence $\al[ h_{\infty,0}]\equiv 0$, contradicting (\ref{eqn_h_infty_l2}).
\end{proof}

The following lemma will reduce the proof of Proposition~\ref{Prop_LRDTF_0T} to two elementary bounds on $\alpha[h]$.

\begin{lemma} \label{lem_reduce_proof}
For any $m \in \IN$, $\eta > 0$ and $D, C' < \infty$ there is a $\delta = \delta (m, \eta, D, C') > 0$ such that the following holds:
If $\alpha[h] (\cdot, -1) \leq C'$ on $M$ and $\alpha[h] (\cdot, -1) \leq \delta$ on $\{ \delta \leq R(\cdot, -1) \leq \delta^{-1} \}$, then the conclusion of Proposition~\ref{Prop_LRDTF_0T} holds.
\end{lemma}

\begin{proof}
As discussed earlier, we may assume that we are either in Case (A) or in Case (B).
Assume that the lemma was wrong for some fixed $m, \eta, D, C'$ and pick a sequence of counterexamples $(h_i)_{t \in [-1, 0]}$ for a sequence $\delta_i \to 0$.
In Case (A) set $\td{h}_{i,t} := h_{i,t} - \nabla^2 w_{i,t}$, where $(w_{i,t})_{t \in [-1,0]}$ is a solution to the heat equation with initial condition $w_{i,-1} =  (\mathcal{F}_j h_{i,-1} ) u_j$, as explained in Lemma~\ref{Lem_alpha_plus_Hessian}.
In Case (B) set $\td{h}_{i, -1} := h_{i,-1}$.
Then $|\td{h}_{i,-1}| \leq C'$ on $M$ and $|\td{h}_{i,-1}| \leq \delta_i$ on $\{ \delta_i \leq R(\cdot, -1) \leq \delta_i^{-1} \}$.
It now follows from a standard limit argument that $\td{h}_{i,0} \to 0$ locally uniformly in $C^m$, which implies assertion (c) of Proposition~\ref{Prop_LRDTF_0T} for $h'_0 = \td{h}_{i,0}$ for large $i$, in contradiction to our assumption.
\end{proof}

Due to Lemma~\ref{lem_reduce_proof}, Proposition~\ref{Prop_LRDTF_0T} is a consequence of the following lemma.

\begin{lemma}
Let $(M, (g_t)_{t \leq 0})$ be a rotationally symmetric $\kappa$-solution diffeomorphic to $\IR^3$ or $S^2 \times \IR$ and let $\delta > 0$.
Then there are constants $C' = C'(g_t), T = T(\delta, (g_t)) < \infty$ such that the following holds.
Suppose that $(h_t)_{t \in [-T,0]}$ is a uniformly bounded solution to the linearized Ricci-DeTurck flow with $|\nabla^m h_{-T}| \leq R^{m/2}$ for all $m = 0, \ldots, 3$.
Assume that the assumptions of Case (A) or (B) hold.
Then
\[ \alpha[h](\cdot, 0) \leq C' \quad \text{on} \quad M \qquad \text{and} \qquad \alpha[h] (\cdot, 0) \leq \delta \quad \text{on} \quad \{ \delta \leq R(\cdot, 0) \leq \delta^{-1} \}. \]
\end{lemma}

\begin{proof}
Note that in the cylindrical case, the lemma is a direct consequence of Lemma \ref{lem_l_large_alpha_small}.
So it suffices to consider the case $M \approx \IR^3$.

Fix $\delta$ and $(M, (g_t)_{t\leq 0})$ for the remainder of the proof.
The constant $T$ will be determined in the end of this proof.
Assume that $(h_t)_{t \in [-T,0]}$ is given and consider the isometric $O(3)$-action on $(M, (g_t)_{t\leq 0})$.

\begin{Claim1}
For any $\Theta < \infty$ there is a constant $C^* = C^*(\Theta) < \infty$ such that for all $(x,t) \in M \times [-T,0]$ with $t - \Theta R^{-1} (x,t) \leq -T$ we have $\alpha [h] (x,t) \leq C^*$.
\end{Claim1}

\begin{proof}
Fix some $E' > 1$ and choose $H=H(E')$ according to Proposition~\ref{Prop_weighted_semilocal_max}.
Let $a = R(x,t)$ and consider the quantity $Q' := e^{-Ha t} \frac{|h|}{R^{E'} + a^{E'}}$.
Choose $(y,s) \in M \times [-T, 0]$ such that $2Q'(y,s) > S := \sup_{M \times [-T,0]} Q'$.
Then, by Proposition~\ref{Prop_weighted_semilocal_max} we have
\[ S < 2Q'(y,s) \leq \frac{2}{100} S + 2C \sup_M Q'(\cdot, -T), \]
where $C= C(E')$.
So $S \leq 4C \sup_{M} Q'(\cdot, -T)$.
This implies that for any $(x',t') \in M \times [-T,t]$
\[ \frac{|h|(x',t')}{2 R^{E'} (x',t')} \leq 4C e^{H\Theta} \cdot \frac{1}{R^{E'}(x',t')} \]
So $|h| \leq 8 C e^{H\Theta}$ on $M \times [-T, t]$.
Using the derivative bounds of $h$ at time $-T$ and standard local derivative estimates (see for example \cite[Lemma~A.14]{BamlerKleiner2017}), we can upgrade this bound to a derivative bound at time $t$ and therefore, we obtain a bound on $\alpha[h](x,t)$.
\end{proof}

Fix some arbitrary constant $E > 1$ and let $A < \infty$ be a constant that will be determined in the following claim.
Consider the following quantity on $M \times [-T,0]$
\[ Q := \frac{\alpha [h]}{((|t|+A) R)^{E} + 1}. \]

\begin{Claim2}
There are constants $\Theta  = \Theta (E), A = A(E) < \infty$ and  $\ov{c} = \ov{c}(E) > 0$  such for any $(x,t) \in M \times [-T,0]$ with $t^* := t - \Theta R^{-1} (x,t) \geq -T$ and $c \in [0, \ov{c}]$ we have
\begin{equation}
\label{qrc_decay}
 (QR^c)(x,t) \leq \frac1{10} \sup_{M  } (QR^c) (\cdot, t^*) .
\end{equation}
\end{Claim2}

\begin{proof}
The constant $\Theta < \infty$ will determined in the end of the proof, depending only on $E$.
The constant $\ov{c}$ will be determined in the course of the proof, depending only on $E$ and $\Theta$.
Assume that the statement was wrong for fixed $\Theta$, choose $A_i \to \infty$ and consider solutions $(h_{i,t})_{t \in [-T_i,0]}$ to the linearized Ricci-DeTurck flow, as well as points $(x_i,t_i) \in M \times [-T_i,0]$ where (\ref{qrc_decay}) is violated.
By linearity we may assume without loss of generality that
\begin{equation}
\label{eqn_al_h_i_normalization}
 \alpha[h_i](x_i,t_i) = 1 
\end{equation}
Set $K_i := R (x_i, t_i)$ and $t^*_i :=  t_i - \Theta K_i^{-1} \geq -T_i$.
In Case (A) set
\[ \td{h}_{i,t} := h_{i,t} - \nabla^2 w_{i,t}, \]
where $(w_{i,t})_{t \in [t^*_i,0]}$ is a solution to the heat equation with initial condition $w_{i, t^*_i} = ( \mathcal{F}_j h_{i, t^*_i}) u_j$, as explained in Lemma~\ref{Lem_alpha_plus_Hessian}.
In Case (B) set $\td{h}_{i,t} := h_{i,t}$.
Then 
$\al[\td h_i](x_i,t_i)= \al[h_i](x_i,t_i)=1$.
So, since (\ref{qrc_decay}) is violated at $(x_i, t_i)$, we have
\begin{equation} \label{eq_RQtstar_RQxiti}
 R^c(\cdot, t^*_i) \, \frac{|\td{h}_{i,t^*_i}|}{((|t^*_i| +A_i) R(\cdot, t^*_i))^E + 1} \leq K_i^c \, \frac{10}{((|t_i|+A_i) K_i )^{E} + 1} . 
\end{equation}

Choose $H = H(E), L=L(E) < \infty$ according to Proposition~\ref{Prop_weighted_semilocal_max} and set
\begin{multline*}
 f_i := R^c \cdot e^{- H (|t^*_i| + A_i)^{-1} (t - t^*_i)}  \frac{|\td{h}_i|}{((|t^*_i| +A_i) R)^E + 1}\\
 =R^c (|t^*_i| +A_i)^{-E}\cdot e^{- H (|t^*_i| + A_i)^{-1} (t - t^*_i)}  \frac{|\td{h}_i|}{ R^E + (|t^*_i| +A_i)^{-E}}.
\end{multline*}
Assume in the following that $c \leq \ol{c}(\Theta, L(E))$ is small enough such that by bounded curvature at bounded distance we have for every $(y,s) \in M \times (-\infty, 0]$,
\begin{equation}
\label{eqn_r_c_bound}
R^c \leq 10 R^c(y,s)\qquad \text{on} \qquad P(y,s, L R^{-1/2} (y,s)).
\end{equation}
Let $(y_i, s_i) \in M \times [t^*_i, t_i]$ be a point such that $S_i := \sup_{M \times [t^*_i, t_i]} f_i \leq 2 f_i (y_i,s_i)$.
Then by Proposition~\ref{Prop_weighted_semilocal_max} (for $a = (|t^*_i| + A_i)^{-1}$) and (\ref{eqn_r_c_bound})
\[ S_i \leq 2 f_i (y_i,s_i) \leq 2 \cdot 10 \big( \tfrac1{100} S_i + C \sup_M f_i (\cdot, t^*_i) \big), \]
for some $C = C(E) < \infty$.
After combining this with (\ref{eq_RQtstar_RQxiti}) and replacing $C$ by $1000 e^H C$, we obtain  that on $M \times [t^*_i, t_i]$ 
\[
R^c \frac{|\td{h}_i|}{((|t^*_i|+A_i) R)^E + 1} 
 \leq  K_i^c \, \frac{ C }{((|t_i|+A_i) K_i)^{E} + 1}.
 \]
 So on $M \times [t^*_i, t_i]$ we have
\begin{equation}
\begin{aligned}
\label{eqn_final_tilde_h}
 |\td{h}_i| 
 &\leq C   \bigg( \frac{K_i}{R} \bigg)^c  \frac{((|t_i| + \Theta K_i^{-1} +A_i) R)^E + 1}{((|t_i|+A_i) K_i)^{E} + 1}  \\
 &=  C   \bigg( \frac{K_i}{R} \bigg)^c\frac{((|t_i|  +A_i) R + \Theta K_i^{-1} R)^E + 1}{((|t_i|+A_i) K_i)^{E} + 1} \\
 &=  C   \bigg( \frac{K_i}{R} \bigg)^c\frac{((|t_i|  +A_i) K_i \cdot K_i^{-1} R + \Theta K_i^{-1} R)^E + 1}{((|t_i|+A_i) K_i)^{E} + 1}.
\end{aligned}
\end{equation}

After passing to a subsequence, we may assume that the limit $Z := \lim_{i \to \infty} (|t_i| + A_i) K_i \in [0, \infty]$ exists or is infinite.
Let us now consider the parabolically rescaled pointed flows $(M, (g_{i,t} := K_i g_{t_i + K_i^{-1} t})_{t \leq 0}, x_i)$.
By the compactness theory of $\kappa$-solutions and after passing to a subsequence, we may assume that these pointed flows smoothly converge to a pointed $\kappa$-solution $(M_\infty, \lb (g_{\infty, t})_{t \leq 0}, \lb x_\infty)$ with $R(x_\infty , 0) = 1$.
By (\ref{eqn_final_tilde_h}), the correspondingly rescaled flows $(\td{h}'_{i,t} := K_i \td{h}_{i,t_i + K_i^{-1} t})_{t \in [-\Theta, 0]})$ satisfy a bound of the form
\begin{equation} \label{eq_tdhprime_rescaled}
 \big| \td{h}'_i \big|_{g_i} \leq C   R^{-c} \frac{((|t_i|  +A_i) K_i \cdot  R + \Theta  R)^E + 1}{((|t_i|+A_i) K_i)^{E} + 1}. 
\end{equation}
Here the scalar curvature is taken with respect to the rescaled metrics.
Since the right-hand side converges to a finite limit, the sequence $\td{h}'_i$ is locally uniformly bounded, so after passing to a subsequence, we may assume that it smoothly converges to a linearized Ricci-DeTurck flow $(\td{h}_{\infty, t})_{t \in (-\Theta, 0]}$ on $M_\infty \times (- \Theta, 0]$ with
\begin{equation}
\label{eqn_al_tilde_h_infty}
\al[\td h_\infty](x_\infty,0)=1.
\end{equation}

\medskip
\textit{Case 1: $Z = \lim_{i \to \infty} (|t_i| + A_i) K_i = \infty$. \quad}
Then passing (\ref{eq_tdhprime_rescaled}) to the limit, we get that
\[ |\td{h}_\infty| \leq CR^{-c} R^E \]
on $M_\infty\times (-\Theta,0]$.
Assume that $2c <  E-1$.
For every $\de'>0$, if $\Theta \geq \underline{\Theta}(\de',E)$, we may apply Corollary~\ref{cor_interior_decay} with $E$ replaced by $E-c$ to obtain that
\[
|\td h_\infty| \leq \delta' C  R^{-c} R^E \qquad \text{on} \qquad P(x_\infty, 0, 1), \qquad m = 0, \ldots, (\delta')^{-1}.
\]
When $\de'$ smaller than some constant depending only on $C= C(E)$ and $E$, we may deduce bounds on $|\nabla^m \td{h}^\infty| (x_\infty, 0)$ that contradict (\ref{eqn_al_tilde_h_infty}).

\medskip
\textit{Case 2: $Z = \lim_{i \to \infty} (|t_i|+A_i) K_i < \infty$. \quad}  
We claim that in this case $(M_\infty, \lb (g_{\infty,t})_{t \leq 0}, \lb x_\infty)$ must be isometric to the standard round shrinking cylinder with $R (\cdot, t) = (1+2|t|)^{-1}$.
Assume not.
Then $\sup_M R(\cdot, t_i) / K_i$ would be uniformly bounded and therefore $(|t_i|+A_i) \sup_M R(\cdot, t_i)$ would be uniformly bounded as well.
By (\ref{eq_extra_assumption}) this would imply that $|t_i|$ remains bounded.
However, since $A_i \to \infty$, we also must have $\sup_M R(\cdot, t_i) \to 0$, contradicting the fact that $|t_i|$ remains bounded.

Passing (\ref{eq_tdhprime_rescaled}) to the limit, we get that
\[
 |\td{h}_\infty| \leq   C  R^{-c} \frac{((Z+\Theta)R)^E +1}{Z^E+1}
 = C  (1+2|t|)^{c} \frac{((Z+\Theta)(1+2|t|)^{-1})^E +1}{Z^E+1}\,.
\]
Assume in the following that $c \leq \ov{c} (\Theta)$ such that $(1+2\Theta)^c\leq 2$.
Then
\begin{multline*}
 \limsup_{t\searrow -\Theta}|\td{h}_{\infty,t}| \leq C   (1+2\Theta)^{c} \frac{((Z+\Theta)(1+2\Theta)^{-1})^E +1}{Z^E+1} \\
 \leq 2C   \frac{(Z+1)^E +1}{Z^E+1}
 \leq  2C \frac{2^E Z^E + 2^E+1}{Z^E+1} \leq 2C (2^E+1).
\end{multline*}
By Lemma~\ref{lem_l_large_alpha_small}, if $\Theta$ is larger than some constant depending on $C = C(E)$, then $\al[\td h_\infty](x_\infty,0)<\frac12$, contradicting (\ref{eqn_al_tilde_h_infty}).
\end{proof}

By combining Claim~1 with Claim~2 for $c = 0$ and observing that $Q \leq \alpha[h]$, we obtain that $Q \leq  C^*$ on $M \times [-T,0]$.
Therefore, we have
\[ \alpha[h](\cdot, 0) \leq C^* \big( (AR(\cdot,0))^E + 1 \big), \]
which implies the first bound of the lemma for some $C' = C'(E)$, since $R$ is uniformly bounded.
In order to prove the second bound, it suffices to show that $Q$ can be made arbitrarily small on $\{ \delta \leq R(\cdot, 0) \leq \delta^{-1} \}$ if $T$ is chosen sufficiently large.
To see this choose $x_0 \in \{ \delta \leq R(\cdot, 0) \leq \delta^{-1} \}$ and assume that $Q(x_0, 0) \geq \delta' > 0$.
We can inductively construct a sequence of points $(x_0, t_0) = (x_0, 0), (x_1, t_1), \ldots$ by the following algorithm:
If $t_{i+1} := t_i - \Theta R^{-1} (x_i, t_i) < T$, then stop the algorithm at $(x_i, t_i)$.
Otherwise, use Claim~2 with $c = \ov{c}$ to find a point $x_{i+1} \in M$ with
\[ ( R^c Q) (x_{i+1}, t_{i+1}) \geq 10 (R^c Q)(x_{i}, t_i). \]
So if $(x_i, t_i)$ is defined, then $(R^c Q)(x_i, t_i) \geq \delta'$ and
\[ (R^c Q)(x_0, 0) \leq 10^{-i} (R^c Q)(x_i, t_i) \leq 10^{-i} \max_{M \times (-\infty,0]} R^c C^*. \]
It remains to show that the sequence $(x_i, t_i)$ exist for large enough $i$ if $T$ is chosen sufficiently large.
To see this, note that whenever $(x_i, t_i)$ is defined, we have
\[ t_i - t_{i+1} =  \Theta R^{-1}(x_{i}, t_{i})    \leq \Theta \bigg( \frac{Q(x_{i}, t_{i})}{(R^c Q)(x_i,t_i)} \bigg)^{1/c} \leq \Theta \bigg( \frac{C^*}{\delta'} \bigg)^{1/c}. \]
So for fixed $i$ and sufficiently large $T$, we have $t_{i+1} \geq -T$ and the algorithm can be continued.
\end{proof}

\section{The main argument} \label{sec_main_argument}
\subsection{Choice of constants and terminology}
In the following, we will define the \emph{scale of a point $x$} by $\rho(x) := |{\Rm}|^{-1/2} (x) \in (0, \infty]$.

We will first fix some constants, which we will use throughout this section.
Let $\underline{E} < \infty$ be the constant from \cite[Theorem~1.7]{BamlerKleiner2017} (Strong Stability of Ricci flow Spacetimes) and fix $E > \underline{E}$.
Based on this choice, let $\eps_{\can} := \eps_{\can} (E) > 0$, again according to \cite[Theorem~1.7]{BamlerKleiner2017}.

We will now fix a constant $m_{\can} \in \IN$ according to the following (trivial) Lemma.

\begin{Lemma} \label{Lem_mcan_def}
There are constants $D_{\can} < \infty$, $m_{\can} \in \IN$ and $\eps'_{\can} > 0$ such that the following holds.
Let $(M,g, x)$ be a pointed Riemannian manifold and $(\ov{M}, \ov{g}, \ov{x})$ the pointed final time-slice of a $\kappa$-solution.
Assume that there is a diffeomorphism onto its image
$$ \psi : B^{\ov{M}} (\ov{x}, D_{\can}) \to M $$
such that $\psi (\ov{x}) = x$  and such that for $\lambda := \rho (x)$
$$ \big\Vert \lambda^{-2} \psi^* g - \ov{g} \big\Vert_{C^{m_{\can}} (B^{\ov{M}} (\ov{x}, D_{\can}))} < \eps'_{\can}. $$
Then $(M,g)$ satisfies the $\eps_{\can}$-canonical neighborhood assumption at $x$.
\end{Lemma}

Let $D$ be a constant, whose value will be determined later in Lemma~\ref{Lem_gluing}.
Using this constant and the constant $m_{\can} \in \IN$, we now define a quantity $\beta$ that measures the degree to which a metric is locally $O(3)$-invariant.

\begin{Definition}[Pointed roundness] \label{Def_pointed_roundness}
Let $(M,g,x)$ be a complete, pointed Riemannian manifold.
We define $\beta (x)$ to be the infimum over all $\beta' > 0$ with the following property:
There is a pointed Riemannian manifold $(\ov{M}, \ov{g}, \ov{x})$ that admits an isometric $O(3)$-action whose generic orbits are 2-spheres, and such that $B^{\ov{M}} (\ov{x}, D)$ is relatively compact, as well as a diffeomorphism onto its image
$$ \psi : B^{\ov{M}} (\ov{x}, D) \longrightarrow M $$
such that $\psi (\ov{x}) = x$ and such that for $\lambda := \rho (x)$
$$ \big\Vert \lambda^{-2} \psi^* g - \ov{g} \big\Vert_{C^{m_{\can} + 100} (B^{\ov{M}} (\ov{x}, D))} < \beta'. $$
If $(M, g)$ is the time-$t$-slice of a Ricci flow, then we will write $\beta(x,t)$ instead of $\beta(x)$. 
\end{Definition}

Note that $\beta$ is an upper semi-continuous function.
A standard limit argument yields

\begin{Lemma} \label{Lem_alpha_0_O3}
If $\beta \equiv 0$ on $M$, then $(M,g)$ admits an isometric $O(3)$-action whose generic orbits are 2-spheres.
\end{Lemma}

We will moreover use the following asymptotic roundness property of $\kappa$-solutions.

\begin{Lemma} \label{Lem_sup_alpha_to_0}
Let $(M, (g_t)_{t \leq 0})$ be a $\kappa$-solution on $\IR^3$ or $S^3$.
In the case $M \approx S^3$, we additionally assume that Theorem~\ref{Thm_O3_invariance_general} already holds for any $\kappa$-solution on $\IR^3$.

Then there is a sequence $t_i \searrow - \infty$ such that $\sup_M \beta (\cdot, t_i) \to 0$.
\end{Lemma}

This lemma is the same as \cite[Lemma~2.7]{Brendle2018}.

\begin{proof}
The case $M \approx \IR^3$ is a consequence of the rigidity discussion of Hamilton's Harnack inequality \cite{Hamilton1993} and Brendle's uniqueness result of the Bryant soliton among $\kappa$-solutions that are solitons \cite{Brendle2013}.
The proof is the same as in \cite{Brendle2018}, so we omit it here.

The case $M \approx S^3$ follows from the fact that the flow is either homothetic to the shrinking round sphere or any rescaling limit for $t \searrow - \infty$ is a shrinking round cylinder or is diffeomorphic to $\IR^3$ and therefore rotationally symmetric by assumption.
\end{proof}

Lastly, we will also use

\begin{Lemma} \label{Lem_gluing}
There are universal constants $D, C_0 < \infty$ with the following property.
Let $(M,g)$ be complete 
and $\sup_M\al  \leq \beta^* \leq C_0^{-1}$.
Then there is a complete $O(3)$-invariant metric $g'$ on $M$ whose generic orbits are 2-spheres such that for all $m = 0, \ldots, m_{\can} + 90$
\[ \big| \rho^m \nabla^m (g' - g) \big| \leq C_0 \beta^*. \]
\end{Lemma}

\begin{proof}
Standard gluing argument.
\end{proof}

\subsection{The main stability estimate}
Our main estimate will be the following proposition.

\begin{Proposition} \label{Prop_alpha_decay}
Given any $\kappa$-solution $(M, (g_t)_{t \leq 0})$ that is not a constant curvature space form, we can find constants $\ov\beta > 0$, $A < \infty$ such that the following holds for any $(x,t) \in M \times (- \infty, -A]$.
If $\beta \leq \ov\beta$ on $M \times [ t - A R^{-1} (x,t), t]$, then
\[ \beta (x,t) \leq \frac1{10} \sup_{M \times [  t - A R^{-1} (x,t), t]} \beta . \]
\end{Proposition}

\begin{proof}
Fix $(M, (g_t)_{t \leq 0})$ and choose $\ov\beta_i \to 0$, $A_i \to \infty$.
Assume that the statement of the proposition was wrong and choose a sequence of counterexamples $(x_i, t_i) \in M \times (-\infty, -A_i]$ such that $\beta \leq \min \{ \ov\beta_i, 10 \beta (x_i,t_i) \}$ on $M \times [t- A_i R^{-1} (x_i, t_i), t_i]$.
Let $g_{i,t} := R (x_i, t_i) g_{t_i + R^{-1} (x_i, t_i) t}$ be the parabolically rescaled flow on which $R(x_i, 0) = 1$.
We will only work with the pointed sequence of $\kappa$-solutions $(M, (g_{i,t})_{t \leq 0},x_i)$ from now on.
After passing to a subsequence, we may assume that the pointed flows $(M, (g_{i,t})_{t \leq 0}, x_i)$ smoothly converge to some limiting pointed $\kappa$-solution $(M_\infty, (g_{\infty, t})_{t \leq 0}, x_\infty)$.
This limit is non-compact since $t_i \leq -A_i \searrow -\infty$ and $(M, (g_t)_{t \leq 0})$ was assumed not to have constant sectional curvature.
Moreover, $(M_\infty, (g_{\infty, t})_{t \leq 0}, x_\infty)$ is $O(3)$-invariant, because $\ov\beta_i \to 0$.

For each $i$ let
\[ \beta^*_i := \sup_{M \times [ - A_i, 0]} \beta_{g_i} \leq \ov\beta_i \longrightarrow 0. \]
Let $T  > 0$ be a constant whose value we will determine later.
By Lemma~\ref{Lem_gluing}, we can find for large $i$ a complete $O(3)$-invariant metric $g'_{i,-T}$ on $M$ such that for $m = 0, \ldots, m_{\can} + 90$
\begin{equation} \label{eq_g_prime_m_g}
 \big| \rho^{m} (\cdot, -T) \nabla^m (g'_{i,-T} - g_{i, -T}) \big| \leq C_0 \beta^*_i 
\end{equation}
Recall that $C_0$ is a universal constant.
Let $\M'_i$ be the Ricci flow spacetime with initial condition $(M, g'_{i,-T})$ on the time-interval $[-T,0]$.
More specifically, we require $\M'_i$ to be $0$-complete and satisfy the $\eps'$-canonical neighborhood assumption below some small enough scale for any $\eps' > 0$.
The existence of $\M'_i$ is guaranteed by \cite{Kleiner:2014le}.

We will now compare $(M, (g_{i,t})_{t \in [-T,0]})$ with $\M'_i$ and express $\M'_i$ as a Ricci-DeTurck flow on the background $(g_{i,t})_{t \in [-T,0]}$, after modification by a family of diffeomorphisms.
Unfortunately, both flows may a priori differ significantly far away from $x_i$, so we will only be able to express $\M'_i$ as a Ricci-DeTurck flow in a large parabolic neighborhood around $x_i$.

In the following, we will apply the Strong Stability Theorem \cite[Theorem~1.7]{BamlerKleiner2017} to compare $(M, (g_{i,t}))$ with the Ricci flow spacetime $\M'_i$.
Note that the former can be viewed as a Ricci flow spacetime, as explained in \cite[sec~5]{BamlerKleiner2017}.
Fix some arbitrary number $\delta > 0$ and choose $\eps = \eps(\delta, T, E)$, where $T$ is the constant from this proof.
Then we can find a sequence of scales $r_i \to 0$ such that 
\[ C_0 \beta^*_i = \eps r_i^{2E}.  \]
Set (with respect to $g_{i,-T}$)
\[ U_i := \big\{ |{\Rm_{g_i}}| (\cdot, -T) <  (\eps r_i)^{-2} \big\} \subset M. \]
Then on $U_i$
\[ |g_{i,-T} - g'_{i,-T} | \leq C_0 \beta^*_i = \eps r_i^{2E}  < \eps r_i^{2E} \big( |{\Rm_{g_i}}|(\cdot, -T) + 1 \big)^E. \]
For each $i$ choose $t^*_i \in [-T, 0]$ maximal such that $\M'_i$ restricted to the time-interval $[-T, t^*_i)$ satisfies the $\eps_{\can}$-canonical neighborhood assumption at scales $< 1$.
Set $t^*_i := -T$, if no such maximum exists.
Note that, since $(M, (g_{i,t}))$ is a $\kappa$-solution, it satisfies the $\eps_{\can}$-canonical neighborhood assumption at all scales by definition.
So all assumptions of the Strong Stability Theorem hold on the time-interval $[-T,t^*_i)$ for $\phi = \id_U$.
By the Strong Stability Theorem, we then obtain for each $i$ a subset $\widehat{U}_i \subset M \times [-T, t^*_i)$ such that $|{\Rm_{g_i}}| \geq r_i^{-2}$ on $M \times [-T, t^*_i) \setminus \widehat{U}_i$ and a time-preserving diffeomorphism $\widehat{\phi}_i : \widehat{U}_i \to \M'_i$ such that
\[  h_{i,t} := \widehat{\phi}_{i,t}^* g'_{i,t} - g_{i,t} \]
evolves by Ricci-DeTurck flow on $M \times [-T, t^*_i)$ and such that
\begin{equation} \label{eq_h_bound}
 |h_{i,t}| \leq \delta r_i^{2E} ( |{\Rm_{g_i}}| + 1 )^E = \frac{\delta C_0}{\eps} \cdot \beta_i^* ( |{\Rm_{g_i}}| + 1 )^E. 
\end{equation}

\medskip
\textit{Case 1: $t^*_i = 0$ for large $i$. \quad}
After passing to a subsequence, $h_{i,t} / \beta_i^*$ converges to a linearized Ricci-DeTurck flow $(\td{h}_{\infty, t})_{t \in [-T,0]}$ on the background flow $(M_\infty, (g_{\infty, t})_{t \leq 0}, x_\infty)$, which is uniformly bounded due to (\ref{eq_h_bound}).
By standard local parabolic derivative estimates (see for example \cite[Lemma~A.14]{BamlerKleiner2017}) and Arzela-Ascoli, we may assume that the convergence $h_{i,t} / \beta_i^* \to \td{h}_{\infty, t}$ is locally smooth on $M_\infty \times (-T,0]$ and is locally $C^{m_{\can} + 80}$ on $M_\infty \times [-T,0]$.
Moreover, by (\ref{eq_g_prime_m_g}) we have for $m = 0, \ldots m_{\can}+80$ 
\[ \rho^{m} (\cdot, -T) \big| \nabla^m \td{h}_{\infty, -T} \big| \leq C_0.  \]

Let $\eta > 0$ be a constant whose value will be determined later and apply Proposition~\ref{Prop_LRDTF_0T}, assuming that $T$ is large enough depending on $\eta, D, m_{\can}, C_0$.
This yields the following decomposition on $B(x_\infty, 0, 2 D)$
\[ \td{h}_{\infty, 0} = h^{\rot}_0 + \nabla^2 f + h'_{\infty, 0}, \]
where $h^{\rot}_0$ is rotationally symmetric with respect to the standard $O(3)$-action on $(M, g_{\infty, 0})$, $f \in C^\infty ( B(x_\infty, 0, 2 D) )$ and $\Vert h'_{\infty, 0} \Vert_{C^{m_{\can}+100} (B(x_\infty, 0, 2 D))} \leq \eta$.

Define the maps $\chi_i : B(x_\infty, 0, 2D) \to M_\infty$ by $\chi_i (z) := \exp_z (\beta^*_i \cdot \frac12 \nabla f_z)$.
Since $\beta^*_i \to 0$, the restrictions $\chi_i |_{B(x_\infty, 0, 1.9 D)}$ are diffeomorphisms onto their images for large $i$, which smoothly converge to the identity.
Define the metrics 
\[ g^*_i := \chi_i^* ( g_{\infty,0} + \beta^*_i h^{\rot}_0 ). \]
on $B(x_\infty, 0, 1.9 D)$.
Then, as $i \to  \infty$
\[ \frac1{\beta^*_i} \big( g^*_i - ( g_{\infty,0} + \beta^*_i h^{\rot}_0 ) \big) 
= \frac1{\beta^*_i} \big( \chi^*_i g_{\infty, 0} - g_{\infty, 0} \big) + \chi^*_i h^{\rot}_0 - h^{\rot}_0 \longrightarrow \mathcal{L}_{\frac12 \nabla f} g_{\infty, 0} = \nabla^2 f.
\]
It follows that
\begin{multline*}
 \frac1{\beta^*_i} \big( \widehat{\phi}_{i,0}^* g'_{i,0} - g^*_{i} \big) = \frac1{\beta^*_i} \big( \widehat{\phi}_{i,0}^* g'_{i,0} - (g_{\infty, 0} + \beta^*_i h^{\rot}_0) \big) - \frac1{\beta^*_i} \big( g^*_{i}  - (g_{\infty, 0} + \beta^*_i h^{\rot}) \big) \\
 = \frac1{\beta^*_i} h_{i,0} - h^{\rot}_0 - \nabla^2 f \longrightarrow h'_{\infty, 0}.
\end{multline*}
So, assuming that $\eta$ is smaller than some universal constant, we obtain that $\beta (x_i,0) \leq \frac1{10} \beta^*_i$ for sufficiently large $i$, in contradiction to our choice of $x_i$.

\medskip
\textit{Case 2: After passing to a subsequence $t^*_i < 0$ for all $i$. \quad}
Recall that for each $i$ the flow $\M'_i$ satisfies the $\eps_{\can}$-canonical neighborhood assumption below some positive scale $r'_i > 0$.
So by the maximality of $t^*_i$ and an openness argument, we can find a point $y_i \in \M'_{i, t^*_i}$ of scale $\rho (y_i) < 2$ in the time-$t^*_i$-slice that violates the $\eps_{\can}/2$-canonical neighborhood assumption.
After rescaling the flow $\M'_i$ parabolically by $\rho^{-2} (y_i)$ and applying a time-shift so that the point $y_i$ is contained in the time-$0$-slice and has scale $1$, we obtain a sequence of singular flows $\M''_{i}$ on time-intervals of the form $[-T^*_i, 0]$ that satisfy the $\eps_{\can}$-canonical neighborhood assumption at scales $< 2 / \rho_i$.

\medskip
\textit{Case 2a: After passing to a subsequence, $T^*_\infty := \lim_{i \to \infty} T^*_i$ exists. \quad}
Similarly as in Case~1, we can apply the Strong Stability Theorem to compare the each flow $\M''_i$ with the correspondingly parabolically rescaled and time-shifted flow $(M, ( \rho_i^2 g_{i, t^*_i + \rho_i^{-2} t})_{t \in [-T^*_i, 0]})$.
Since (\ref{eq_g_prime_m_g}) remains preserved under rescaling and $\beta^*_i \to 0$, the Strong Stability Theorem yields that a larger and larger neighborhood of $y_i$ in $\M''_{i,0}$ becomes closer and closer to an open subset in $(M, \rho_i^2 g_{i, t^*_i})$ in the $C^{m_{\can} + 80}$-sense.
Since $(M, \rho_i^2 g_{i, t^*_i})$ is a time-slice of a $\kappa$-solution, this contradicts the choice of $y_i$ for large $i$ via Lemma~\ref{Lem_mcan_def}.

\medskip
\textit{Case 2b: $T^*_i \to \infty$. \quad}
In this case we must have $\lim_{i \to \infty} \rho_i = 0$.
Assuming that $\eps_{\can}$ is smaller than some universal constant, we can argue as in \cite[12.1]{Perelman1} to show that, after passing to a subsequence, the pointed flows $(\M''_i,y_i)$ smoothly converge to a pointed ancient non-singular flow $(M''_\infty, (g''_{\infty, t})_{t \leq 0}, y_\infty)$ with non-negative sectional curvature that satisfies the $2\eps_{\can}$-canonical neighborhood assumption at all scales.
Therefore $(M''_\infty, (g''_{\infty, t})_{t \leq 0}, y_\infty)$ is a $\kappa$-solution, in contradiction to the choice of $y_i$ for large $i$.
\end{proof}

\subsection{Proof of Theorem~\ref{Thm_O3_invariance_general}}

\begin{proof}[Proof of Theorem~\ref{Thm_O3_invariance_general}.]
It suffices to consider the case in which $(M, (g_t)_{t \leq 0})$ is not the quotient of a round sphere or a round cylinder.
Therefore $M$ is must be diffeomorphic to $\IR^3, S^3$ or $\IR P^3$.
By passing to the double cover, the case $\IR P^3$ can be reduced to the case $S^3$.
So, we only need to consider the case in which $(M, (g_t)_{t \leq 0})$ is diffeomorphic to $\IR^3$ or $S^3$, but not the shrinking round sphere.
By proving the theorem first in the $\IR^3$-case, we may moreover assume that the theorem is already true in this case when proving the case $M \approx S^3$.
Thus Lemma~\ref{Lem_sup_alpha_to_0} will be applicable in both cases.

Define $\beta : M \times (-\infty,0] \to \IR$ as in Definition~\ref{Def_pointed_roundness} and let $\ov\beta, A$ be the constants from Proposition~\ref{Prop_alpha_decay}.

Choose $\beta' \in (0, \ov\beta]$.
We will show in the following that $\beta \leq \beta'$ on $M \times (-\infty,0]$.
By letting $\beta' \to 0$, this will imply that $\beta \equiv 0$, which implies rotational symmetry by Lemma~\ref{Lem_alpha_0_O3}.

By Lemma~\ref{Lem_sup_alpha_to_0} there is a sequence $t_i \searrow -\infty$ such that
\begin{equation} \label{eq_alpha_ti_to_0}
\sup_M \beta(\cdot, t_i) \longrightarrow 0.
\end{equation}
Fix some large $i$ for which $\sup_M \beta(\cdot, t_i) \leq \beta'$ and choose $t^*_i \geq t_i$ maximal such that $\beta \leq \beta'$ on $M \times [t_i, t^*_i)$.

If $t^*_i = 0$ for infinitely many $i$, then $\beta \leq \beta'$ everywhere and we are done.
So assume that $t^*_i < 0$ for large $i$.
In the following we will only consider such indices $i$.
By maximal choice of $t^*_i$ and the upper semi-continuity of $\beta$, there is a point $y_i \in M$ such that $\beta (y_i, t^*_i) \geq \beta' / 2$.

Next, we argue that
\[ t^*_i - t_i < A R^{-1} (y_i, t_i). \]
In fact, if the opposite inequality were true, then we could apply Proposition~\ref{Prop_alpha_decay} (recall that $\beta' \leq \ov\beta$) and conclude that $\beta (y, t^*_i) \leq \beta' / 10$, in contradiction to the choice of $y_i$.

Let now $Q_i := R(y_i, t^*_i)$.
After passing to a subsequence, we may assume that $T := \lim_{i \to \infty} (t^*_i - t_i) Q_i$ exists and that the pointed and parabolically rescaled flows $(M, (Q_i g_{t^*_i + Q_i^{-1} t})_{t \leq 0}, y_i)$ converge to a pointed $\kappa$-solution $(M_\infty, (g_{\infty, t})_{t \leq 0}, y_\infty)$.
By (\ref{eq_alpha_ti_to_0}) we obtain that $g_{\infty, -T}$ is rotational symmetric.
So $g_{\infty, 0}$ must be rotational symmetric as well, in contradiction to the choice of $y_i$ for large $i$.
\end{proof}

\bibliography{library}{}
\bibliographystyle{amsalpha}

\end{document}